\documentclass{article}
\usepackage{amsthm, amsfonts, amssymb,latexsym, color}
\usepackage{graphicx, float}
\title{Incidences of Cubic Curves in Finite Fields}
\author{Audie Warren}
\newcommand{\F}{\mathbb{F}}

\usepackage{amsmath}
\usepackage[a4paper, total={6in, 8in}]{geometry}
\newtheorem{proposition}{Proposition}

\numberwithin{exercise}{subsection}
\newtheorem{lemma}{Lemma}

\newtheorem{theorem}{Theorem}
\newtheorem*{theorem*}{Theorem}

\newtheorem{corollary}{Corollary}

\begin{document}
 \maketitle
\begin{abstract}
In this paper we prove an incidence bound for points and cubic curves over prime fields. The methods generalise those used by Mohammadi, Pham, and Warren \cite{conicincidences}.
\end{abstract}
 \section{Introduction}
Given a set of points $P$ in the plane $\mathbb F^2$ over a field $\mathbb F$, and a set of irreducible algebraic curves $C$ in $\mathbb F^2$, the number of incidences between $P$ and $C$ is defined as 
$$I(P,C) := \{(p,\gamma) \in P \times C: p\in\gamma \}.$$
In the case $\mathbb F = \mathbb R$ and when $C$ is actually a set of lines $L$, an optimal upper bound for $I(P,L)$ was given by Szemer\'{e}di and Trotter \cite{Szemeredi1983}.
\begin{theorem}[Szemer\'{e}di-Trotter]\label{thm:sztr}
For any finite sets of points and lines $P$ and $L$ in the real plane, we have\footnote{In this paper we use the notation $A \ll B$ to mean that there exists an absolute constant $c>0$ such that $A \leq cB$. We have $B \gg A$ if  $A \ll B$.}
$$I(P,L) \ll (|P||L|)^{2/3} + |P|+ |L|.$$
\end{theorem}
Over $\mathbb R$, this theorem has been generalised to other curves, the most well known such result being the Pach-Sharir theorem, see \cite{pachsharir1992}, \cite{pachsharir}. Such results for algebraic curves have also been proven over the complex numbers, see \cite{sheffer}.

In this paper we consider the case $\mathbb F = \mathbb F_p$ for prime $p$. In this setting, point-line incidence bounds analogous to Theorem \ref{thm:sztr} are known, the first such result being proved by Bourgain, Katz, and Tao \cite{BKT}. The state of the art point-line incidence bound is due to Stevens and de Zeeuw \cite{StevensdeZeeuw} - which itself relies on the point-plane incidence bound of Rudnev \cite{Rudnev}. Given that the sets of points and lines are not too large with respect to the characteristic $p$, they give the bound
\begin{equation}\label{sdzbound}
    I(P,L) \ll (|P||L|)^{11/15} + |P| + |L|.
\end{equation}
Using the basic geometric fact that two lines intersect in one point, and two points define one line, one can apply the K\H{o}v\'{a}ri-S\'{o}s-Tur\'{a}n theorem \cite{KST} to the incidence graph of $P$ and $L$ to obtain
$$I(P,L) \ll \min\{|P||L|^{1/2}+|L|, |P|^{1/2}|L| + |P| \}.$$
The bound \eqref{sdzbound} improves upon these bounds for a certain balancing of $|P|$ and $|L|$.

Obtaining incidence bounds between points and non-linear algebraic curves in $\mathbb F_p$ has proved a difficult task, with very few results being known. However, recently there has been a flurry of activity concerning incidences between points and certain degree two curves in $\mathbb F_p$, see for instance \cite{mishajames} and \cite{mobius}. Pushing the methods used in these papers further, an incidence bound between points and arbitrary irreducible conics was given in a paper of Mohammadi, Pham, and Warren \cite{conicincidences}.

In this paper, we adapt and generalise ideas present in \cite{conicincidences} to prove an incidence bound between points and arbitrary cubic curves in $\mathbb F_p$. Our main result is the following.

\begin{theorem}\label{thm:main1}
Let $P$ be a set of points in $\mathbb F_p^2$, with $|P|\leq p^{15/13}$, and let $C$ be any set of irreducible cubic curves in $\mathbb F_p^2$. Then we have
$$I(P,C) \ll \min\{(|P||C|)^{39/43},|P||C|^{9/10},|P|^{1/2}|C|\} + |P|+|C|.$$
\end{theorem}

In fact, we will prove the following bound.

\begin{theorem}\label{thm:main2}
Let $P$ be a set of points in $\mathbb F_p^2$, with $|P|\leq p^{15/13}$, and let $C$ be any set of irreducible cubic curves in $\mathbb F_p^2$. Then we have
$$I(P,C) \ll (|P||C|)^{39/43} + |P|^{71/43}|C|^{28/43}+|C|.$$
\end{theorem}
It is again important to compare this result with the trivial bounds given by K\H{o}v\'{a}ri-S\'{o}s-Tur\'{a}n. As above, this is given by the basic fact that two irreducible cubic curves intersect in at most nine points. This yields
$$I(P,C) \ll \min \{ |P||C|^{9/10} + |C|, |P|^{1/2}|C| + |P| \}.$$
Comparing these bounds to the first term in Theorem \ref{thm:main2}, we see that Theorem \ref{thm:main2} improves upon the trivial bounds when we have
$$|P|^{35/8} \leq |C| \leq |P|^{40/3},$$
and within this range the second term of Theorem \ref{thm:main2} is dominated by the first. Theorem \ref{thm:main1} is then the augmentation of Theorem \ref{thm:main2} with the K\H{o}v\'{a}ri-S\'{o}s-Tur\'{a}n bounds. We note that although we have focused on $\F_p$, the results extend to other fields, with the same restriction on the size of $P$ with respect to the characteristic $p$, and also to fields of characteristic zero by ignoring the restriction on the characteristic.

We mention that it is crucial to restrict to irreducible curves in Theorem \ref{thm:main1} (and such incidence results in general), as otherwise $I(P,C) = |P||C|$ is obtainable. Take a single line $l$, and let all of $P$ lie on $l$. Define a set of reducible cubic curves $C$, where each is the union of $l$ with some other conic. Since every point lies on $l$, which is a component of every cubic in $C$, the number of incidences is $|P||C|$. 

\section{Proof of Theorem \ref{thm:main2}}
\subsection{The set-up}
We now begin the proof of Theorem \ref{thm:main2}. The main idea will be to, in a certain sense, dualise the points and curves $P$ and $C$, so that we recover point and line incidences. However, we will not work with incidences directly, choosing to instead work with $k$-rich curves. A curve $\gamma \in C$ is called $k$-rich if it contains between $k$ and $2k$ points of $P$, that is,
$$k \leq |\gamma \cap P| < 2k.$$
We let $C_k \subseteq C$ be the set of $k$-rich curves from $C$. Our main goal will be to bound, for all $k$ sufficiently large, $|C_k|$. This will be achieved by first considering the problem locally.

Let $S \subseteq P$ be a set of seven points. We make the definition
$$C_{k,S} := \{ \gamma \in C_k : \forall q \in S, q \in \gamma  \}.$$
In words, this is the set of $k$-rich curves which pass through all points of $S$. Given a bound for each $C_{k,S}$, we can give a bound on $C_k$. Indeed, if we sum over all subsets $S \subseteq P$ of size seven, each $k$-rich curve will be counted at least ${k \choose 7} \gg k^7$ times, noting that this assumes $k \geq 7$. This implies that we have the inequality
\begin{equation}\label{eqn:kcounting}
|C_k| \ll \frac{1}{k^7}\sum_{\substack{S \subseteq P \\ |S| = 7}} |C_{k,S}|.\end{equation}
We now begin the main part of the proof, which is to bound $|C_{k,S}|$.
\subsection{Bounding $C_{k,S}$}
To begin the dualisation process, we provide a map $\phi$ which sends our curves $C$ to points in $\mathbb P(\mathbb F_p^9)$. The map is very simple - it takes a curve $f(x,y)=0$ to its list of coefficients. Note that this is a map into projective space since constant multiples of an equation $f(x,y)=0$ determine the same curve.
$$\phi: \{\text{Curves of degree at most 3 over } \mathbb F_p^2 \} \longrightarrow \mathbb P(\mathbb F_p^9) $$
$$\sum_{\substack{(i,j) \\ i+j \leq 3}} c_{i,j}x^iy^j=0 \longrightarrow [c_{0,0}: c_{0,1}: ... :c_{2,1}: c_{3,0}].$$
The ordering chosen for the coordinates is irrelevant - we simply fix an ordering and use it consistently.

Fix a point $q = (q_1,q_2) \in \mathbb F_p^2$. If we let $\Gamma_q$ be the set of all degree at most 3 curves passing through $q$, then the image $\phi(\Gamma_q)$ is a hyperplane in $\mathbb P(\mathbb F_p^9)$, since the point $q$ imposes a single linear condition on the coefficients of the curves. Indeed, the points $ [X_{0,0}: X_{0,1}: ... :X_{2,1}: X_{3,0}] \in \phi(\Gamma_q)$ are precisely those that satisfy the linear equation
$$\sum_{\substack{(i,j) \\ i+j \leq 3}} X_{i,j}q_1^iq_2^j=0.$$
We denote such a hyperplane by $\pi_q$. We now take our set $S\subseteq P$ of size seven, and look at the image under $\phi$ of \emph{all} degree at most 3 curves which pass through the points of $S$, call them $\Gamma_S$. From the above, this is given by
$$\phi(\Gamma_S) = \bigcap_{q \in S} \pi_q.$$
We prove a lemma to control this image.
\begin{lemma} \label{lem:intersection}
Let $S\subseteq P$ be a set of seven points. Then either $\phi(\Gamma_S)$ is a $2$-flat, or $C_{k,S}$ is the empty set.
\end{lemma}
In order to prove this, we require a simple proposition. The following is a version of a result present in \cite{cayleybacharach} - a proof can be found there which is valid over sufficiently large fields.
\begin{proposition} \label{prop:fourcollinear}
Let $S$ be a set of points in $\mathbb F_p^2$. 
\begin{itemize}
    \item If $|S|=7$ and $S$ contains no five collinear points, then $S$ imposes independent conditions on the set of all cubic curves.
    \item If $|S|=8$ and $S$ contains no five collinear points and are not all one a common conic, then $S$ imposes independent conditions on the set of all cubic curves.
\end{itemize}
\end{proposition}
The statement ``$S$ imposes independent conditions on the set of all cubic curves'' means that the intersection $\cap_{q \in S}\pi_q$ is complete, that is, has dimension two. Note that the only way this can fail to happen is if at some point one of these intersections were trivial, that is, a hyperplane $\pi_q$ contains the previous intersections $\cap_{q' \in S'}\pi_{q'}$ for some subset $S' \subset S$. If this happens, then every cubic curve passing through all of $S'$ also passes through $q$. We can now prove Lemma \ref{lem:intersection}. 

\begin{proof}[Proof of Lemma \ref{lem:intersection}]
Note that if $S$ were contained in a conic, we must have $C_{k,S} = \emptyset$, as otherwise this conic intersects an irreducible cubic curve in seven points. This implies that $\Gamma_S$ contains only cubic curves.  If $S$ contains four collinear points, then $S$ cannot be contained within any irreducible cubic curve, by Bezout's theorem, and therefore $C_{k,S} = \emptyset$.  On the other hand, if no four points of $S$ are collinear, then by Proposition \ref{prop:fourcollinear}, the intersections of the hyperplanes $\pi_q$ for $q \in S$ is complete, so that $\phi(\Gamma_S)$ is a $2$-flat.
\end{proof}
We continue the proof, assuming that $|C_{k,S}| \neq 0$, implying that $\phi(\Gamma_S)$ is a $2$-flat. Let $\pi_S$ denote this $2$-flat. We have that $\phi(C_{k,S}) \subseteq \pi_S$, and $\pi_S$ contains only points corresponding to cubic curves.

The next step is to give a map which sends our original points $P$ to lines in $\pi_S$. Since points not lying on any curve from $C_{k,S}$ do not contribute any incidences, we only perform this step for points which do indeed lie on curves from $C_{k,S}$ - by an abuse of notation we denote such points by $P \cap C_{k,S}$. Furthermore, we ignore the points of $S$, as they would be, in a certain sense, degenerate for this map. We define the map as follows.
$$\psi: (P \cap C_{k,S})\setminus S \rightarrow \{\text{lines in }\pi_S\}$$
$$\psi(q) = \pi_q \cap \pi_S.$$
We must justify, firstly, that $\psi(q)$ is indeed a line in $\pi_S$. Since we are intersecting a hyperplane with a $2$-flat, $\psi(q)$ can either be a line, as needed, or we have $\pi_q \cap \pi_S = \pi_S$. If this second case were to occur, it would mean that $\pi_S \subseteq \pi_q$, so that $S \cup \{q\}$ does \emph{not} impose independent conditions on cubic curves, which by Proposition \ref{prop:fourcollinear} implies that it contains five collinear points, or all eight are on a conic. In the first case, by removing $q$ we find at least four points of $S$ collinear, contradicting the assumption $|C_{k,S}| \neq 0$. In the second case, we must have that $S$ lies on a conic, again contradicting Bezout's Theorem unless $C_{k,s} = \emptyset$. We therefore conclude that $\psi(q)$ is indeed a line.

Secondly, we check the multiplicity of the lines $\psi(q)$. We claim that for each line $l$ lying in $\pi_S$, there are at most two points $q,q'$ which are both mapped to $l$, that is, these lines are defined with multiplicity at most two. To prove this, suppose there exist three points $q_1,q_2,q_3$ with $\psi(q_1)=\psi(q_2)=\psi(q_3)=:l$. Consider the set $S \cup \{q_1,q_2,q_3\}$. Since $q_1,q_2,q_3 \in (P \cap C_{k,S})\setminus S$, there must exist an irreducible cubic curve $\gamma \in C_{k,S}$ such that $\phi(\gamma) \in l$. Indeed, this follows since we have for all $q \in (P \cap C_{k,S})\setminus S$, and $\gamma \in C_{k,S}$,
$$q \in \gamma \iff \phi(\gamma) \in \psi(q).$$ Then $\gamma$ contains the ten points $S \cup \{q_1,q_2,q_3\}$. On the other hand, since $l$ is a line, we can take any point other than $\phi(\gamma)$ on $l$, and we find another (possibly reducible) cubic curve containing $S \cup \{q_1,q_2,q_3\}$. Since $\gamma$ is irreducible, this contradicts Bezout's theorem.

We now put together all of the above information, to recover an incidence problem between points and lines in $\mathbb F_p^2$. Take a $k$-rich curve $\gamma \in C_{k,S}$. It has been mapped to a point $\phi(\gamma) \in \pi_S$. Each point $q \in P \setminus S$ which lies on $\gamma$ has been sent, via $\psi$, to a line $\psi(q) \subseteq \pi_S$, and this line must contain the point $\phi(\gamma)$, since $q \in \gamma$. Such lines are defined with multiplicity at most two. Therefore, the $k$-rich curve $\gamma$ has been sent to an at least $\frac{k-7}{2}$-rich \emph{point} $\phi(\gamma)$, with respect to the lines $L := \psi((P \cap C_{k,S})\setminus S)$. We can now bound $|C_{k,S}|$ by the number of $\frac{k-7}{2}$-rich points defined by a set of $|L| \leq |P|$ lines in $\mathbb F_p^2 \cong \pi_S$. This is done via the following result of Stevens and de Zeeuw \cite{StevensdeZeeuw}.
\begin{corollary}\label{cor:krichpoints}
Let $L$ be a set of lines in $\F_p^2$, with $|L|\ll p^{15/13}$, and for $t\geq 2$ let $P_t$ denote the number of $t$-rich points with respect to $L$. Then
\[
|P_t|\ll \frac{|L|^{11/4}}{t^{15/4}} + \frac{|L|}{t}.
\]
\end{corollary}
Note that this is where the condition $|P|\ll p^{15/13}$ is adopted. Since we are applying this result with $t = \frac{k-7}{2}$, we must assume $k \geq 11$. This gives
$$|C_{k,S}| \ll \frac{|P|^{11/4}}{k^{15/4}} + \frac{|P|}{k}.$$
\subsection{Finishing the proof}
Returning to equation \eqref{eqn:kcounting}, we can bound the number of $k$-rich curves for $k \geq 11$ as 
$$|C_k| \ll \frac{|P|^{9 + 3/4}}{k^{10 + 3/4}} + \frac{|P|^8}{k^8}.$$
We can now follow a standard argument to bound $I(P,C)$. In the following we denote by $C_{=k}$ the set of \emph{precisely} $k$-rich curves.
\begin{align*}
    I(P,C) &= \sum_{k\geq 1} |C_{=k}|k \\
    & = \sum_{k \leq \Delta} |C_{=k}|k + \sum_{k > \Delta} |\mathcal{C}_{=k}|k \\   & \ll \Delta|C| + \sum_{i\geq 0}|C_{2^i\Delta}|(2^{i} \Delta) \\
    & \ll \Delta |C| + \sum_{i\ge 0}\left(\frac{|P|^{9+3/4}}{(2^{i} \Delta)^{10+3/4}} + \frac{|P|^8}{(2^{i} \Delta)^8}\right)(2^{i} \Delta) \\
    & \ll \Delta |C| + \frac{|P|^{9+3/4}}{\Delta^{9+3/4}} + \frac{|P|^8}{\Delta^7}.
\end{align*}
We now optimise our choice of $\Delta$. In order to ensure that the application of Corollary \ref{cor:krichpoints} was valid, we must have $\Delta \geq 11$. The best choice is then
$$\Delta= \max \left\{ 11, \frac{|P|^{39/43}}{|C|^{4/43}} \right\}.$$
If the second term is taken in this maximum, we recover the first two terms of Theorem \ref{thm:main2}. If the first term is chosen, then we must have $|C|^4 \gg |P|^{39}$, and in this case our bound gives $I(P,C) \ll |C|$. Combining these two possibilities yields Theorem \ref{thm:main2}.

\section*{Acknowledgements}
The author was partially supported by the Austrian Science Fund FWF
Project P-34180. I thank Niels Lubbes, Mehdi Makhul, Oliver Roche-Newton, Josef Schicho, and Ali Uncu for very helpful conversations.
\bibliography{cubicpaper}{}
\bibliographystyle{amsplain}
\end{document}